\documentclass[a4paper,11pt]{amsart}

\usepackage{epsfig}
\usepackage{amsthm,amsfonts}
\usepackage{amssymb,graphicx,color}
\usepackage[all]{xy}
\usepackage{verbatim}
\usepackage{hyperref}
\usepackage{enumitem}

\newtheorem{theorem}{Theorem}[section]
\newtheorem*{theorem*}{Theorem}
\newtheorem{lemma}[theorem]{Lemma}
\newtheorem{corollary}[theorem]{Corollary}
\newtheorem{proposition}[theorem]{Proposition}
\newtheorem{definition}[theorem]{Definition}
\newtheorem{conjecture}{Conjecture}
\newtheorem{example}[theorem]{Example}
\newtheorem{remark}[theorem]{Remark}
\newtheorem{claim}{Claim}[theorem]

\newcommand{\R}{\mathbb{R}}
\newcommand{\C}{\mathbb{C}}

\newcommand{\N}{\mathbb{N}}

\begin{document}

\title[On the Fukui-Kurdyka-Paunescu
Conjecture]
{On the Fukui-Kurdyka-Paunescu
Conjecture}

\author[A. Fernandes]{Alexandre Fernandes}
\author[Z. Jelonek]{Zbigniew Jelonek}
\author[J. E. Sampaio]{Jos\'e Edson Sampaio}

\address[Alexandre Fernandes and Jos\'e Edson Sampaio]{    
              Departamento de Matem\'atica, Universidade Federal do Cear\'a,
	      Rua Campus do Pici, s/n, Bloco 914, Pici, 60440-900, 
	      Fortaleza-CE, Brazil. \newline  
              E-mail: {\tt alex@mat.ufc.br}\newline  
              E-mail: {\tt edsonsampaio@mat.ufc.br}
}
\address[Zbigniew Jelonek]{ Instytut Matematyczny, Polska Akademia Nauk, \'Sniadeckich 8, 00-656 Warszawa, Poland. \newline  
              E-mail: {\tt najelone@cyf-kr.edu.pl}
}

\keywords{Zariski's Multiplicity Conjecture, Fukui-Kurdyka-Paunescu's Conjecture, Arc-analytic, Bi-Lipschitz homeomorphism, Degree, Multiplicity}
\subjclass[2010]{14B05, 32S50, 58K30 (Primary) 58K20 (Secondary)}
\thanks{The first named author was partially supported by CNPq-Brazil grant 304221/2017-1. The second named author is partially supported by the grant of Narodowe Centrum Nauki number 2019/33/B/ST1/00755. The third named author was partially supported by CNPq-Brazil grant 303811/2018-8.
}
\begin{abstract}
In this paper, we prove Fukui-Kurdyka-Paunescu's Conjecture, which says that subanalytic arc-analytic bi-Lipschitz homeomorphisms preserve the multiplicities of  real analytic sets. We also prove  several other results on the invariance of the multiplicity (resp. degree) of real and complex analytic (resp. algebraic) sets.  For instance, still in the real case, we prove a global version of Fukui-Kurdyka-Paunescu's Conjecture.
In the complex case, one of the results that we prove is the following: If $(X,0)\subset (\mathbb{C}^n,0), (Y,0)\subset (\mathbb{C}^m,0)$ are germs of analytic sets and $h\colon (X,0)\to (Y,0)$ is a semi-bi-Lipschitz homeomorphism whose graph is a complex analytic set, then the germs $(X,0)$ and $(Y,0)$ have the same multiplicity. One of the results that we prove in the global case is the following: If $X\subset \mathbb{C}^n, Y\subset \mathbb{C}^m$ are algebraic sets and 
$\phi\colon X\to Y$ is a semialgebraic semi-bi-Lipschitz homeomorphism such that the closure of its graph in $\mathbb{P}^{n+m}(\mathbb{C})$ is an orientable homological cycle, then ${\rm deg}(X)={\rm deg}(Y)$.
\end{abstract}

\maketitle

\section{Introduction}

Zariski's famous  Multiplicity Conjecture, stated by Zariski in 1971 (see \cite{Zariski:1971}), is formulated as follows: 
\begin{conjecture}[Zariski's Multiplicity Conjecture]\label{zariski}
Let $f,g \colon(\C^n,0)\to (\C,0)$ be two reduced complex analytic functions. If there is a homeomorphism $\varphi\colon(\C^n,V(f),0)\to (\C^n,V(g),0)$, then $m(V(f),0)=m(V(g),0)$.
\end{conjecture} 
This is still an open problem; see \cite{Eyral:2007} for a survey on this conjecture. 
% , for example, Zariski in \cite{Zariski:1932} gave a positive answer when $n=2$ and 
In the real case, of course, Zariski's Multiplicity Conjecture does not hold in the same form as in the complex case. However, we have the following conjecture, stated by Fukui, Kurdyka and Paunescu  \cite[Conjecture 3.3]{FukuiKP:2004}:

\begin{conjecture}[Fukui-Kurdyka-Paunescu's Conjecture] Let $X,Y\subset\R^n$ be two germs at the origin of irreducible real analytic subsets. If $h\colon(\R^n,0)\to (\R^n,0)$ is the germ of a subanalytic, arc-analytic and bi-Lipschitz homeomorphism such that $h(X)=Y$, then $m(X,0)\equiv m(Y,0) \,{\rm mod\,} 2$.
\end{conjecture}

Several authors approached this conjecture: For example, J.-J. Risler  \cite{Risler:2001} proved that multiplicity $\,{\rm mod\,} 2$ of a real analytic curve is invariant under bi-Lipschitz homeomorphisms; T. Fukui, K. Kurdyka and L. Paunescu  \cite{FukuiKP:2004} confirmed the conjecture  in the case that $X$ and $Y$ are real analytic curves; 
and G. Valette \cite{Valette:2010} showed that  multiplicity $\,{\rm mod\,} 2$ of real analytic hypersurfaces is invariant under arc-analytic bi-Lipschitz homeomorphisms. The third named author of this paper proved in \cite{Sampaio:2020b} the real version of Gau-Lipman's theorem: i.e., multiplicity $\,{\rm mod\,} 2$ of real analytic sets is invariant under homeomorphisms $\varphi\colon (\R^n,0)\to (\R^n,0)$ such that $\varphi$ and $\varphi^{-1}$ have a derivative at the origin. A generalization of this result was presented in \cite{Sampaio:2021b}.

In this paper, we give a complete, positive answer to Fukui-Kurdyka-Paunescu's Conjecture (see Theorem \ref{proof_FKP-conjecture}). A global version of this conjecture is also proved (see Corollary \ref{proof_global_FKP}).

Coming back to the complex case, let us list some contributions to Zariski's Multiplicity Conjecture from the Lipschitz point of view. For instance, Neumann and Pichon  \cite{N-P}, with previous contributions of Pham and Teissier  \cite{P-T} and Fernandes  \cite{F}, proved that the bi-Lipschitz geometry of plane curves determines the Puiseux pairs, and as a consequence  if two germs of complex analytic curves with any codimension are bi-Lipschitz homeomorphic (with respect to the outer metric), then they have the same multiplicity. Comte  \cite{Comte:1998} proved that  multiplicity of complex analytic germs (not necessarily codimension 1 sets) is invariant under bi-Lipschitz homeomorphisms with the severe assumption that the Lipschitz constants are close enough to 1. This motivated the following conjecture in \cite{BobadillaFS:2018}:
\begin{conjecture} \label{conj_local}
Let $X\subset \C^n$ and $Y\subset \C^m$ be two complex analytic sets with $\dim X=\dim Y=d$. If their germs at zero are bi-Lipschitz homeomorphic, then their multiplicities $m(X,0)$ and $m(Y,0)$ are equal.
\end{conjecture}
In \cite{BobadillaFS:2018}the following conjecture  was also posed:
\begin{conjecture}\label{conj_infinity}
Let $X\subset \C^n$ and $Y\subset \C^m$ be two complex algebraic sets with $\dim X=\dim Y=d$. 
If $X$ and $Y$ are bi-Lipschitz homeomorphic at infinity, then  ${\rm deg}(X)={\rm deg}(Y)$.
\end{conjecture}
Still in \cite{BobadillaFS:2018} the authors proved that Conjectures \ref{conj_local} and \ref{conj_infinity} are equivalent and, moreover, have positive answers for $d=1$ and $d=2$. However, Birbrair et al.  \cite{BirbrairFSV:2020} disproved  these conjectures when $d\ge 3$, by showing explicit counter-examples. More precisely, it was shown that we have two different embeddings of   $\mathbb{P}^{1}(\C)\times \mathbb{P}^{1}(\C)$ into $\mathbb{P}^{5}(\C)$, say $X$ and $Y$, such that their affine cones $\textup{Cone}(X),\textup{Cone}(Y)\subset \C^6$ are bi-Lipschitz equivalent, but they have different degrees. Hence the problem of invariance of degree under bi-Lipschitz homeomorphisms is still open in the important case of affine hypersurfaces in $\C^n$ with $n>3.$ Moreover, there are several cases where Conjectures \ref{conj_local} and \ref{conj_infinity} hold true, for instance, the Lipschitz Regularity Theorem  \cite[Theorem 4.2]{Sampaio:2016} (see also \cite{BirbrairFLS:2016}) shows that if a germ of an analytic set is bi-Lipschitz equivalent to a smooth germ, then it is smooth itself, which implies that  multiplicity $1$ of a complex analytic germ  is a bi-Lipschitz invariant. Fernandes and Sampaio  \cite{FernandesS:2020} proved that  degree $1$ of a complex algebraic set is invariant under bi-Lipschitz homeomorphism at infinity and Sampaio in \cite{Sampaio:2019} proved the version of Comte's result for the degree: the degree of complex algebraic sets is invariant under bi-Lipschitz homeomorphism at infinity with Lipschitz constant close enough to 1. Recently, Jelonek  \cite{Jelonek:2021} proved that the multiplicity of complex analytic sets is invariant under bi-Lipschitz homeomorphisms which have analytic graphs, and the degree of complex algebraic sets is invariant under bi-Lipschitz homeomorphisms (at infinity) which have algebraic graph.

In this paper, we prove some generalizations of the results proved by Jelonek  \cite{Jelonek:2021}. For instance, we show that the multiplicity of complex analytic sets is invariant under semi-bi-Lipschitz homeomorphisms which have analytic graph (see Theorem \ref{inv_mult}) and the degree of complex algebraic sets is invariant under semi-bi-Lipschitz homeomorphisms at infinity which have algebraic graph (see Theorem \ref{inv_degree_semi}). We also prove  that degree of a complex algebraic set is invariant under semialgebraic semi-bi-Lipschitz homeomorphisms at infinity such that the closure of their graphs are  orientable homological cycles (see Theorem \ref{inv_degree_homo_cycle}).

\section{Preliminaries}\label{section:preliminaries}

\subsection{Lipschitz and semi-bi-Lipschitz mappings}

\begin{definition}
Let $X\subset \R^n$ and $Y\subset \R^m$ be two sets and let $h\colon X\to Y$.
\begin{itemize}
 \item We say that $h$ is {\bf Lipschitz} if there exists a positive constant $C$ such that
$$\|h(x)-h(y)\|\leq C\|x-y\|, \quad \forall x, y\in X.
$$ 
\item We say that $h$ is {\bf bi-Lipschitz}  if $h$ is a homeomorphism, it is Lipschitz and its inverse  is also Lipschitz.
\item We say that $h$ is {\bf bi-Lipschitz at infinity} (resp. {\bf a homeomorphism at infinity}) if there exist compact subsets $K\subset \R^n$ and $K'\subset \R^m$ such that $h|_{X\setminus K}\colon X\setminus K\to Y\setminus K'$ is bi-Lipschitz (resp. a homeomorphism).
 \item We say that $h$ is {\bf semi-Lipschitz at $x_0\in X$} if there exist a positive constant $C $  such that
$$
\|h(x)-h(x_0)\|\leq C\|x-x_0\|, \quad \forall x\in X.
$$
\item  We say that $h$ is {\bf semi-bi-Lipschitz}  if $h$ is a homeomorphism, it is semi-Lipschitz at $x_0$ and its inverse  is also semi-Lipschitz at $h(x_0)$.
\item We say that $h$ is {\bf semi-bi-Lipschitz at infinity} if there exist compact subsets $K\subset \R^n$ and $K'\subset \R^m$ such that $h|_{X\setminus K}\colon X\setminus K\to Y\setminus K'$ is semi-bi-Lipschitz at some point $x_0\in X\setminus K$.
\end{itemize}
\end{definition}

Now we give a geometric characterization of semi-bi-Lipschitz mappings. For a similar characterization of bi-Lipschitz mappings see \cite{Jelonek:2021}.

\begin{definition}
Let $L^s, H^{n-s-1}$ be two disjoint linear subspaces of $\mathbb{P}^{n}(\C).$ Let $\pi_\infty$ be a hyperplane (a hyperplane at infinity) and assume that
$L^s \subset \pi_\infty.$ The projection $\pi_L$ with center $L^s$ is the mapping
$$ \pi_L \colon \C^n=\mathbb{P}^{n}(\C)\setminus \pi_\infty\ni x \mapsto \langle L^s,x \rangle \cap H^{n-s-1}\in H^{n-s-1}\setminus \pi_\infty=\C^{n-s-1}.$$ Here  $\langle L,x\rangle$
we mean the linear projective subspace spanned by $L$ and $\{x\}.$
\end{definition}

\begin{lemma}\label{lemat}
Let $X$ be a closed subset of $\C^n.$ Denote by $\Lambda_0\subset \pi_\infty$ the set of directions of all secants of $X$ which contain $x_0$ and let $\Sigma_0= \overline{\Lambda_0},$
where $\pi_\infty$ is the hyperplane at infinity and we consider the euclidean  closure. Let $\pi_L:\C^n\to \C^l$ be the projection with center $L$. Then $\pi_L|_X$ is  semi-bi-Lipschitz at $x_0$ if and only if $L\cap \Sigma_0=\emptyset.$
\end{lemma}

\begin{proof}
a) Assume that $L\cap \Sigma_0=\emptyset.$ We will proceed by induction.
Since a linear affine isomorphism is a bi-Lipschitz homeomorphism, we can assume that $\pi_L$ coincides with the projection
$\pi\colon \C^n\ni (x_1,\ldots,x_n) \mapsto (x_1,\ldots,x_k,0,\ldots,0)\in \C^k\times \{0,\ldots,0\}.$ We can decompose $\pi$ into two projections:
$\pi=\pi_{2}\circ \pi_{1}$, where $\pi_1\colon\C^n \ni (x_1,\ldots,x_n)\mapsto (x_1,\ldots,x_{n-1},0)=\C^{n-1}\times 0$ is the projection with  center $P_1=(0:0:\ldots:1)$
and $\pi_2 \colon\C^{n-1}\ni (x_1,\ldots,x_{n-1},0)\mapsto (x_1,\ldots,x_{k},0,\ldots,0)\in \C^{k}\times \{(0,\ldots,0)\}$ is the projection with  center $L':=\{ x_0=0,\ldots,x_{k}=0\}.$ Since $P_1\in L$ and consequently $P_1\not\in \Sigma$, we will prove  that $\pi_1$ is a semi-bi-Lipschitz homeomorphism.
Indeed,  $P_1\in \mathbb{P}^{n-1}(\C)\setminus \Sigma_0.$   We show that the projection $p=\pi_1|_X\colon X\to \C^{n-1}\times 0$  is  semi-bi-Lipschitz.
Of course $\|p(x)-p(x_0)\|\le \|x-x_0\|.$ Assume that $p$ is not semi-bi-Lipschitz, i,e., there is a sequence of points $x_j\in X$
such that $$\frac{\|p(x_j)-p(x_0)\|}{\|x_j-x_0\|}\to 0$$ as $j\to \infty.$ Let $x_j-x_0=(a_1(j),\ldots,a_{n-1}(j),b(j))$ and denote by
$P_j$ the corresponding point $(a_1(j):\ldots:a_{n-1}(j):b(j))$ in $\mathbb{P}^{n-1}(\C).$ Hence $$P_j=\frac{(a_1(j):\ldots:a_{n-1}(j):b(j))}{\|x_j-x_0\|}.$$ Since $\frac{(a_1(j),\ldots,a_{n-1}(j))}{\|x_j-x_0\|}=
\frac{p(x_j)-p(x_0)}{\|x_j-x_0\|}\to 0$, we see that $P_j\to P.$ It is a contradiction.
Notice that if $\pi_1(X)=X'$, then $\Sigma_0'=\pi_1(\Sigma_0), x_0'=p(x_0).$ Moreover $L'=L\cap \{ x_n=0\}$ and $\langle L', P_1\rangle=L$. This means that $\Sigma_0'\cap L'=\emptyset.$ Now  by induction the projection $\pi_2|_{p(X)}$ is semi-bi-Lipschitz at $p(x_0),$ hence also $\pi=\pi_1\circ\pi_2$ is semi-bi-Lipschitz.

\vspace{5mm}

b) Assume that $\pi_L|_X$ is a semi-bi-Lipschitz mapping and $\Sigma_0 \cap L\not=\emptyset.$ As before we can change the system of coordinates in such a way that
$\pi\colon \C^n\ni (x_1,\ldots,x_n) \mapsto (x_1,\ldots,x_k,0,\ldots,0)\in \C^k\times \{0,\ldots,0\}.$ Moreover, we can assume that $P_1=(0:0:\ldots:1)\in \Sigma_0.$
Actually $\pi_1$ is not semi-bi-Lipschitz. Indeed there is a sequence of secants $l_n=\langle x_n,x_0\rangle$ of $X$ whose directions tend to $P_1.$
 Let $x_j-x_0=(a_1(j),\ldots,a_{n-1}(j),b(j))$ and denote by
$P_j$ the corresponding point $(a_1(j):\ldots:a_{n-1}(j):b(j))$ in $\mathbb{P}^{n-1}(\C).$ Hence $$P_j=\frac{(a_1(j):\ldots:a_{n-1}(j):b(j))}{\|x_j-x_0\|}.$$ Since  $P_j\to P$ we have $\frac{(a_1(j),\ldots,a_{n-1}(j))}{\|x_j-x_0\|}=
\frac{p(x_j)-p(x_0)}{\|x_j-x_0\|}\to 0$. Hence the mapping $\pi_1$ is not semi-bi-Lipschitz at $x_0.$ 

Let $x_0'=\pi_1(x_0).$ Now it is enough to note that $\|\pi_2(x)-\pi_2(x_0')\|\le \|x-x_0'\|,$ hence $\|\pi(x_n)-\pi(x_0)\|=\|\pi_2(\pi_1(x_n))-\pi_2(\pi_1(x_0))\|\le \|\pi_1(x_n)-\pi_1(x_0)\|$. Thus  $$\frac{\|x_n-x_0\|}{\|\pi(x_n)-\pi(x_0)\|}\ge \frac{\|x_n-x_0\|}{\|\pi_1(x_n)-\pi_1(x_0)\|}\to \infty.$$ This contradiction finishes the proof.
\end{proof}

\begin{lemma}\label{wykres}
Let $X\subset \C^{n}$ be  a closed  set and let $f\colon X\to\C^m$ be a semi-Lipschitz homeomorphism. Let $Y:=\textup{graph}(f)\subset \C^n\times \C^m.$
Then the mapping $\phi\colon X\ni x \mapsto (x,f(x)) \in Y$ is a semi-bi-Lipschitz homeomorphism.
\end{lemma}

\begin{proof}
Since $f$ is semi-Lipschitz at $x_0$, there is a constant $C$ such that $$\|f(x)-f(x_0)\|\leq C\|x-x_0\|.$$
We have $$\|\phi(x)-\phi(x_0)\|=\|(x-x_0,f(x)-f(x_0))\|$$  $$ \le \|x-x_0\|+\|f(x)-f(x_0)\|\le \|x-x_0\|+C\|x-x_0\|\le (1+C)\|x-x_0\|.$$
Moreover $$\|x-x_0\| \le \|\phi(x)-\phi(x_0)\|.$$ Hence $$\|x-x_0\|\le \|\phi(x)-\phi(x_0)\|\le (1+C)\|x-x_0\|.$$
\end{proof}

\begin{remark}
{\rm It is easy to note that Lemmas \ref{lemat} and  \ref{wykres} hold in the real case also.}
\end{remark}

\begin{definition}
Let $A\subset \R^n$ be a subset. We say that $v\in \R^n$ is a {\bf tangent vector to} $A$  {\bf at $p\in \overline{A}$} (resp. {\bf at infinity}) if there is a sequence of points $\{x_i\}_{i\in \N}\subset A$ such that $\lim\limits_{i\to \infty} \|x_i-p\|=0$ (resp. $\lim\limits_{i\to \infty} \|x_i\|=+\infty $) and there is a sequence of positive numbers $\{t_i\}_{i\in \N}\subset\R^+$ such that 
$$\lim\limits_{i\to \infty} \frac{1}{t_i}(x_i-p)= v\quad \mbox{(resp. } \lim\limits_{i\to \infty} \frac{1}{t_i}x_i= v).$$
Let $C(A,p)$ (resp. $C_{\infty }(A)$) denote the set of all tangent vectors to $A$ at $p$ (resp. at infinity). The subset $C(A,p)$ (resp. $C_{\infty }(A)$) is called {\bf the tangent cone of $A$ at $p$} (resp. {\bf at infinity}).
\end{definition}

\begin{definition}
Let $X\subset \R^n$ and $Y\subset \R^m$ be subanalytic sets with $0\in X$ and $0\in Y$ and let $h\colon (X,0)\to (Y,0)$ be a subanalytic Lipschitz mapping. We define the {\bf pseudo-derivative of $h$ at $0$}, $d_{0}h\colon C(X,0)\to C(Y,0)$, by $d_{0}h(v)=\lim\limits_{t\to 0^+}\frac{h(\gamma(t))}{t}$, where $\gamma\colon [0,+\varepsilon)\to X$ satisfies $\lim\limits_{t\to 0^+}\frac{\gamma(t)}{t}=v$. 
\end{definition}

\begin{definition}
Let $X\subset \R^n$ and $Y\subset \R^m$ be semialgebraic sets and let $h\colon X\to Y$ be a semialgebraic Lipschitz mapping. We define the {\bf pseudo-derivative of $h$ at infinity} $d_{\infty}h\colon C(X,\infty)\to C(Y,\infty)$ by $d_{\infty}h(v)=\lim\limits_{t\to +\infty}\frac{h(\gamma(t))}{t}$, where $\gamma\colon (r,+\infty)\to X$ satisfies $\lim\limits_{t\to +\infty}\frac{\gamma(t)}{t}=v$. 
\end{definition}

\subsection{Multiplicity and degree of real sets}
Let $X\subset \R^{n}$ be a $d$-dimensional real analytic set with $0\in X$ and 
$$
X_{\C}= V(\mathcal{I}_{\R}(X,0)),
$$
where $\mathcal{I}_{\R}(X,0)$ is the ideal in $\mathbb{C}\{z_1,\ldots,z_n\}$ generated by the complexifications of all germs of real analytic functions that vanish on the germ $(X,0)$. We know that $X_{\C}$ is a germ of a complex analytic set and $\dim_{\C}X_{\C}=\dim_{\R}X$ (see \cite[Propositions 1 and 3, pp. 91-93]{Narasimhan:1966}). Then, for a linear projection $\pi:\C^{n}\to\C^d$ such that $\pi^{-1}(0)\cap C(X_{\C},0) =\{0\}$, there exists an open neighborhood $U\subset \C^n$ of $0$ such that $\# (\pi^{-1}(x)\cap (X_{\C}\cap U))$ is constant for a generic point $x\in \pi(U)\subset\C^d$. This number is the multiplicity of $X_{\C}$ at the origin and it is denoted by $m(X_{\C},0)$.

\begin{definition}
With the above notation, we define the multiplicity of $X$ at the origin by $m(X,0):=m(X_{\C},0)$.
\end{definition}

In the same way, we define the degree of a real algebraic set $A\subset \R^n$ by ${\rm deg}(A):={\rm deg}(A_{\C})$, where $A_{\C}=V(\mathcal{I}_{\R}(A))$ and $\mathcal{I}_{\R}(A)$ is the ideal in $\mathbb{C}[z_1,\ldots,z_n]$ generated by the complexifications of all real polynomials that vanish on  $A$.

\begin{definition}
We shall not distinguish between a $2(n-d)$-dimensional real linear subspace in $\C^n$ and its canonical image in the Grassmannian $G^{2n}_{2(n-d)}(\R)$. Thus, we regard the Grassmannian $G^{n}_{n-d}(\C)$ as a subset of $G^{2n}_{2(n-d)}(\R)$.
Let $\mathcal{E}(X_{\C})$ denote the subset of $G^{2n}_{2(n-d)}(\R)$ consisting of all $L\in G^{2n}_{2(n-d)}(\R)$ such that $L\cap C(X_{\C},0)=\{0\}$.
\end{definition}

\begin{remark}[\cite{Sampaio:2020b}]\label{transversal-cones}
We have the following comments on the set  $\mathcal{E}(X_{\C})$.
\begin{enumerate}
 \item [(i)]\label{transversal-cones-i} $\mathcal{E}(X_{\C})$ is an open dense set in $G^{2n}_{2(n-d)}(\R)\cong G^{2n}_{2d}(\R)$.
 \item [(ii)]\label{transversal-cones-ii} For each $L \in \mathcal{E}(X_{\C})\cap G^{n}_{n-d}(\C)$, let $\pi_L\colon \C^n\to L^{\perp}$ be the orthogonal projection along $L$. Then there exist a polydisc $U\subset \C^n$ and a complex analytic set $\sigma\subset U':=\pi_L(U)$ such that $\dim \sigma <\dim X_{\C}$ and $\pi_L\colon (U\cap X_{\C})\setminus \pi_L^{-1}(\sigma)\to U'\setminus \sigma$ is a $k$-sheeted cover with $k=m(X_{\C},0)$.
 \item [(iii)]\label{transversal-cones-iii} Since $\pi:=\pi_L$ is an $\R$-linear mapping, we identify the $d$-dimensional real linear subspace $\pi(\R^n)$ with $\R^d$ and, with this identification, we find that $\R^d\cap \sigma$ is a closed nowhere dense subset of $\R^d\cap U'$. Indeed, it is clear that $\R^d\cap \sigma$ is a closed subset of $\R^d\cap U'$, and thus if $\sigma$ is somewhere dense in $\R^d\cap U'$, then $\sigma$ contains an open ball  $B_r(p)\subset\R^d\cap U'$, which implies that $\sigma$ must contain a non-empty open subset of $U'$  a contradiction. Therefore, $\sigma$ is nowhere dense in $\R^d\cap U'$ and so $\R^d\cap U'\setminus \sigma$ is an open dense subset of $\R^d\cap U'$.
 \item [(iv)]\label{transversal-cones-iv} For a generic point $x\in \R^d$ near  the origin (i.e., for $x\in (\R^d\cap U')\setminus \sigma$), we have
\begin{eqnarray*}
m(X_{\C},0)&=&\# (\pi^{-1}(x)\cap (X_{\C}\cap U))\\
           &=&\# (\R^n\cap \pi^{-1}(x)\cap (X_{\C}\cap U))+\# ((\C^n\setminus \R^n)\cap \pi^{-1}(x)\cap (X_{\C}\cap U))\\
           &=&\# (\pi^{-1}(x)\cap (X\cap U))+\# (\pi^{-1}(x)\cap ((X_{\C}\setminus \R^n)\cap U)).
\end{eqnarray*}
Since for each $f\in \mathcal{I}_{\R}(X,0)$, we may write $f(z)=\sum \limits _{|I|=k}^{\infty}a_Iz^I$ with $a_I\in\R$ for all $I$, it follows that  $f(z_1,\ldots,z_n)=0$ if and only if $f(\bar z_1,\ldots,\bar z_n)=0$, where each $\bar z_i$ denotes the complex conjugate of $z_i$. In particular, $\# (\pi^{-1}(x)\cap ((X_{\C}\setminus \R^n)\cap U))$ is an even number. Therefore,
$m(X,0)\equiv \# (\pi^{-1}(x)\cap (X\cap U)) \,{\rm mod\,} 2$ for a generic point $x\in \R^d$ near  the origin. 
\end{enumerate}
\end{remark}

\begin{definition}
The mapping $\beta _n:\mathbb{S}^{n-1}\times \R^+\to \R^n$ given by $\beta_n(x,r)=rx$ is called the {\bf  spherical blowing-up} (at the origin) of $\R^n$.
\end{definition}
Note that $\beta _n:\mathbb{S}^{n-1}\times (0,+\infty )\to \R^n\setminus \{0\}$ is a homeomorphism with inverse $\beta_n^{-1}:\R^n\setminus \{0\}\to \mathbb{S}^{n-1}\times (0,+\infty )$ given by $\beta_n^{-1}(x)=(\frac{x}{\|x\|},\|x\|)$.
\begin{definition}
The {\bf strict transform} of the subset $X$ under the spherical blowing-up $\beta_n$ is $X':=\overline{\beta_n^{-1}(X\setminus \{0\})}$ and the {\bf boundary} $\partial X'$ of the {\bf strict transform} is $\partial X':=X'\cap (\mathbb{S}^{n-1}\times \{0\})$.
\end{definition}
Note that $\partial X'=C_X\times \{0\}$, where $C_X=C(X,0)\cap \mathbb{S}^{n-1}$.

\begin{definition}
Let $X\subset \R^n$ be a subanalytic set such that $0\in \overline X$ is a non-isolated point. We say that $x\in\partial X'$ is {\bf a simple point of $\partial X'$} if there is an open $U\subset \R^{n+1}$ with $x\in U$ such that:
\begin{itemize}
\item [a)] the connected components of $(X'\cap U)\setminus \partial X'$, say $X_1,\ldots, X_r$, are topological manifolds with $\dim X_i=\dim X$, $i=1,\ldots,r$;
\item [b)] $(X_i\cup \partial X')\cap U$ are topological manifolds with boundary. 
\end{itemize}
Let $\textup{Smp}(\partial X')$ be the set of simple points of $\partial X'$.
\end{definition}
\begin{definition}
Let $X\subset \R^n$ be a subanalytic set such that $0\in X$.
We define $k_X:\textup{Smp}(\partial X')\to \N$, by letting $k_X(x)$ be the number of connected components of the germ $(\beta_n^{-1}(X\setminus\{0\}),x)$.
\end{definition}
\begin{remark}
It is clear that the function $k_X$ is locally constant. In fact, $k_X$ is constant in each connected component $C_j$ of $\textup{Smp}(\partial X')$. Then, we define $k_X(C_j):=k_X(x)$ with $x\in C_j$.
\end{remark}
\begin{remark}\label{density_top}
By Theorems 2.1 and 2.2 in \cite{Pawlucki:1985},  $\textup{Smp}(\partial X')$ is an open dense subset of the $(d-1)$-dimensional part of $\partial X'$ whenever $\partial X'$ is a $(d-1)$-dimensional subset, where $d=\dim X$.
\end{remark}
% \begin{remark}
% The numbers $k_X(C_j)$ are equal to the numbers $n_j$ defined by Kurdyka and Raby \cite{Kurdyka:1989}, pp. 762.
% \end{remark}

\begin{definition}
Let $X\subset \R^{n}$ be a real analytic set. We denote by $C_X'$ the closure of the union of all connected components $C_j$ of $\textup{Smp}(\partial X')$ such that $k_X(C_j)$ is an odd number. We call $C_X'$ the {\bf odd part of $C_X\subset \mathbb{S}^{n}$}.
\end{definition}

\begin{definition}
Let $A\subset \mathbb{R}^{n}$, $B\subset \mathbb{R}^{d}$ and $C\subset A$ be subanalytic sets and $\pi\colon A\to B$ be a continuous mapping. If $\#(\pi^{-1}(x)\cap C)$ is constant ${\rm mod\,} 2$ for a generic $x\in B$, we define the {\bf degree of $C$ with respect to $\pi$} to be ${\rm deg}_{\pi}(C):=\#(\pi'^{-1}(x)\cap C)\ {\rm mod\,} 2$, for a generic $x\in B$.
\end{definition}

Let $X\subset \mathbb{R}^{n}$ be a subanalytic set. If ${\rm deg}_{\pi}(X)$ is defined and does not depend on a generic projection $\pi\colon \R^{n}\to \R^{d}$, then we denote ${\rm deg}_{\pi}(X)$ just by ${\rm deg}_2(X)$.

\begin{remark}
If $X\subset \R^n$ is an algebraic set then ${\rm deg}_2(X)$ is defined and ${\rm deg}_2(X)\equiv {\rm deg}(X)\, mod \, 2$.
\end{remark}

Proposition 2.14 in \cite{Sampaio:2021b} implies the following result:

\begin{proposition}\label{multcone}
Let $X\subset \R^{n}$ be a $d$-dimensional real analytic set with $0\in X$ and $\pi\colon \C^{n}\to \C^{d}$ be a projection such that $\pi^{-1}(0)\cap C(X_{\C},0)=\{0\}$. Let $\pi'\colon\mathbb{S}^n\setminus \pi^{-1}(0)\to \mathbb{S}^{d-1}$ be the mapping given by $\pi'(u)=\frac{\pi(u)}{\|\pi(u)\|}$. Then ${\rm deg}_{\pi'}(C_X')$ is defined and satisfies ${\rm deg}_{\pi}(C_X')\equiv m(X,0) \ {\rm mod\,}\,2$.
\end{proposition}

\section{Homological cycles}
Let $M$ be a smooth compact manifold of (real) dimension $n$.
Given homology classes $\alpha \in H_k(M)$ and $\beta\in H_{n-k}(M)$, we choose representative cycles $\tilde\alpha$ and $\tilde\beta$, respectively.  We can assume that every singular simplex appearing in each of these cycles is a smooth mapping and also that any two simplices meet transversally. This means that the only points of intersection are where the interior of a $k$-simplex in $\tilde\alpha$ meets the interior of an $(n-k)$-simplex in $\tilde\beta$.  At every such point $x$ of intersection both $\tilde\alpha$ and $\tilde\beta$ are local embeddings and their tangent spaces are complementary in $T_xM.$  We assign a sign to each point of intersection by comparing the direct sum of the orientations of the tangent spaces of $\tilde\alpha$ and of $\tilde\beta$ with the ambient orientation of the tangent space of $M.$  The sum of the signs over the (finitely many) points of intersection gives the intersection pairing applied to $(\alpha,\beta).$

If $M=\mathbb{P}^n(\C)$, then $H_{2i}(M,\mathbb{Z})=\mathbb{Z}$ for $i=0,1,\ldots,n$ and $H_{2i-1}(M,\mathbb{Z})=0.$ The space $H_{2i}(M,\mathbb{Z})$ is generated by the class
$L^{n-i}$ where $L$ is a hyperplane, and we consider it as an algebraic cycle. Hence every $2i$-dimensional homological cycle $\alpha$ can be described as $dL^{n-i}.$ We say that the number $|d|$ is {\bf the topological degree of $\alpha$}. Note that if $X\subset M$ is an $i$-dimensional projective subvariety, then the algebraic degree of $X$ coincides with the topological degree.

Similarly, if $M=\mathbb{P}^n(\R)$, then $H_{i}(M,\mathbb{Z}/(2))=\mathbb{Z}/(2)$ for $i=0,1,\ldots,n$. The space $H_{i}(M,\mathbb{Z}/(2))$ is generated by the class $L^{n-i}$ where $L$ is a hyperplane and we consider it as an algebraic cycle. Hence every $i$-dimensional homological cycle $\alpha$ can be described as $dL^{n-i}.$ We say that the number $d$ is {\bf the topological degree mod 2 of $\alpha$}. Note that if $X\subset M$ is an $i$-dimensional projective subvariety, then the algebraic degree mod 2 coincides with the topological degree.

Let $R=\Bbb Z$ or $R=\Bbb Z/(2).$ Let $X$ be a compact semi-algebraic set of dimension $d.$ We say that $X$ is a homological cycle over $R$, if there exists a stratification $\mathcal{S}$ of $X$ such that it gives on $X$ a structure of a $R$-homological $d$-cycle $\alpha$. We say that this cycle is orientable if $R=\Bbb Z$ and  $[\alpha]\not=0$ in $H_d(X,\Bbb Z).$ It is well known that if $X\subset\Bbb P^n(\C)$ is an irreducible algebraic variety, then it is an orientable homological cycle.

The next result generalizes Theorem 3.4 in \cite{Jelonek:2021}.

\begin{theorem}\label{inv_degree_homo_cycle}
Let $X\subset \C^n,Y\subset \C^m$ be complex algebraic varieties of  dimension $d$ and let $h\colon X\to Y$ be a semialgebraic and semi-bi-Lipschitz homeomorphism. Assume that the closure of ${\rm graph}(h)$ in $\mathbb{P}^{n+m}(\C)$ is an orientable homological cycle. Then
${\rm deg} (X)= {\rm deg } (Y).$
\end{theorem}

\begin{proof}
Let $G$ be the closure of $\textup{graph}(h)$ in $\mathbb{P}^{n+m}(\C)$ and $\overline X$ be the  closure of $X$ in $\mathbb{P}^{n}(\C)$.
Let $\alpha$ is a homological cycle of $G.$ First we show that $p_*(\alpha)$ is a fundamental class of $\overline{X},$ where $p:G\to \overline{X}$ is a projection. Denote $A=\overline{X}\setminus X$ and $B=G\setminus {\rm graph}(h).$ Since $\overline{X}\setminus A$ is homeomorphic to $G\setminus B$ we have isomorphism $H(\overline{X},A)\cong H(G,B)$ (see \cite{lam}). Let us consider the following diagram:

\[\xymatrixrowsep{0.3in}
\xymatrixcolsep{0.08in}
\xymatrix{
    0=H_{2d}(B)\ar[d]\ar[r]
        & H_{2d}(G)\ar[d]\ar[r]
        & H_{2d}(G,B)\ar[d]\ar[r] &  H_{2d-1}(B)\ar[d]\\
    0=H_{2d}(A) \ar[r]
        & H_{2d}(\overline{X}) \ar[r]
        & H_{2d}(\overline{X}, A)\ar[r] & H_{2d-1}(A)=0\\ }\]

Since $\overline{X}$ is a projective variety we have $H_{2d}(X,\Bbb Z)=\Bbb Z$ and $H_{2d-1}(A,\Bbb Z)=0.$ In particular $H_{2d}(\overline{X},A,\Bbb Z)=\Bbb Z.$ Since $[\alpha]\not=0$, the diagram above shows that $[p_*([\alpha])=m\beta\not=0.$ Here $\beta$ denotes the funadamental class of $\overline{X}.$ Changing the orientation of $\alpha$ if necessary, we can assume that $m>0.$ Let $L^{n-d}$ be a
linear space which cuts $X$ transversally such that $\#L\cap X={\rm deg}\ X$. It is easy to see that $|p_*([\alpha])\cdot L|\le {\rm deg}\ X.$
This implies that $m=1$ and consequently $p_*([\alpha])=\beta.$

Hence we can take the orientations of $2d$-dimensional simplices in $G$ to be the
same as in $\overline X$ (using a projection $p\colon G\to \overline X)$. Now if a
linear space meets $X$ transversally with $\#L\cap X={\rm deg}\ X$ then also $p^{-1}(L)$ meets $G$ transversally and it is easy to see that the intersection number is $+1$. Indeed, let $M=\textup{Ker} (\pi)$ where $\pi\colon \C^n\times \C^m\to \C^n$ is the projection. Hence $\pi^{-1}(L)=L\times M.$
Fix a canonical orientation on $M$ given by vectors $v_1,\ldots,v_{2m}$, on $L$ given by $w_1,\ldots, w_{2n-2d}$ and on a $2d$-dimensional simplex $\sigma$ in $X$ - $u_1,\ldots,u_{2d}.$ Let $\sigma'$ be a simplex in $G$ such that $\pi(\sigma')=\sigma$ and let $u_1',\ldots,u_{2d}'$ be its orientation at a generic point such that $\pi(u_i')=u_i.$ In particular $u_i'=u_i+\sum_{j=1}^{2m} a_{ij} v_j.$ Hence the cycles $L$ and $X$ at a point of intersection have orientation given by
  $w_1,\ldots,w_{2n-2d}$ , $u_1,\ldots,u_{2d}.$ The cycles $\pi^{-1}(L)$ and $G$ have orientation given by
   $v_1,\ldots,v_{2m}$,  $w_1,\ldots, w_{2n-2d}$ , $u_1',\ldots,u_{2d}'$ (note that $\pi^{-1}(L)=L\times M$). Now it is easy to see that the determinant of the vectors  $v_1,\ldots,v_{2m};w_1,\ldots, w_{2n-2d};u_1',\ldots,u_{2d}'$ is  equal to the product of the determinant of $v_1,\ldots,v_d$
   and the determinant of the  $w_1,\ldots,w_{2n-2d}; u_1,\ldots,u_{2d}$. Consequently, the sign of these two determinant is the same. This means
   that the  orientation of the vectors 
  $w_1,\ldots,w_{2n-2d}$ , $u_1,\ldots,u_{2d}$ is the same as the orientation  of  $v_1,\ldots,v_{2m}$,  $w_1,\ldots, w_{2n-2d}$ , $u_1',\ldots,u_{2d}'$.

Moreover by Lemma \ref{lemat} the set $\overline{\pi^{-1}(L)}\cap G$ has no points at infinity, because otherwise $\overline{L}\cap \overline{X}$ would have points at infinity (the center of the projection is disjoint from $G$).

This means that the
topological degree of $G$ coincides with the algebraic degree of $X$. The same holds for $Y$. Hence ${\rm deg} (X)= {\rm deg} (Y)$.
\end{proof}

In the same way (in fact the proof is simpler, because we do not have to control the orientation) we have:

\begin{theorem}\label{Ztwo_cycle}
Let $X\subset \R^n,Y\subset \R^m$ be real algebraic sets and let $h\colon X\to Y$ be a semialgebraic and semi-bi-Lipschitz homeomorphism. Assume that the projective closure of the graph of $h$ is a $\mathbb{Z}/(2)$ homological cycle. Then
$\textup{deg} (X)= \textup{deg} (Y) \ \textup{mod} \ 2.$
\end{theorem}

\begin{remark}
{\rm In fact  more general statements  are true.  It is easy to see that our proof works if $h$ is semi-bi-Lipschitz at infinity. Indeed, under this assumption, we still have $\overline{\Gamma}\setminus \Gamma \subset \Sigma_0$, where $\Gamma=\textup{graph}(h).$  Moreover, there are sufficiently general linear subspaces $L$ which omit $K$ (or $K'$). }
\end{remark}

\begin{definition}
An $(n-1)$-dimensional subanalytic set $C$ is said to be an {\bf Euler cycle} if it is a closed set and if, for a stratification of $C$ (and hence for any that refines it), the number of $(n-1)$-dimensional strata containing a given $(n-2)$-dimensional stratum in their closure is even.
\end{definition}

\begin{definition}
We say that a set $C \subset \mathbb{R}^n$ is {\bf $a$-invariant} if it is preserved 
by the antipodal mapping (i.e. $a(C) = C$, with $a(x) = -x$).
\end{definition}

\begin{remark}
The closure in $\mathbb{P}^n(\R)$ of an $a$-invariant Euler cycle $C\subset \R^n$ is a $\mathbb{Z}/(2)$ homological cycle.
\end{remark}

We finish this section by recalling the following two definitions:

\begin{definition}
Let $M$ and $N$ be analytic manifolds. Let $X\subset M$ and $Y\subset N$ be analytic subsets. We say that a mapping $f\colon X\to Y$ is {\bf arc-analytic} if for any analytic arc $\gamma\colon (-1,1)\to X$, the mapping $f\circ \gamma$ is an analytic arc as well.
\end{definition}

\begin{definition}(\cite{PA})
We say that $E\subset \mathbb{P}^N(\R)$ is {\bf arc-symmetric} if for any analytic arc $\gamma\colon (-1,1)\to \mathbb{P}^N(\R)$ such that $\gamma((-1,0))\subset E$, we have $\gamma((0,\epsilon))\subset E$, for some $\epsilon>0.$
\end{definition}

\section{Proof of Fukui-Kurdyka-Paunescu's Conjecture}
We start this section by remarking that we cannot expect invariance of multiplicity without mod 2 in Fukui-Kurdyka-Paunescu's Conjecture, as we can see in the next example:

\begin{example}
Consider $X=\{(x,y,z)\in \R^3; z(x^2+y^2)=y^3\}$ and $Y=\{(x,y,z)\in \R^3; z(x^4+y^4)=y^5\}$. Let $h\colon (\R^3,0)\to (\R^3,0)$ be the mapping given by 
$$
h(x,y,z)=\left\{\begin{array}{ll}
                 \left(x,y,z-\frac{y^3}{x^2+y^2}+\frac{y^5}{x^4+y^4}\right)& \mbox{ if }x^2+y^2\not =0\\
                 (0,0,z)& \mbox{ if }x^2+y^2=0.
                \end{array}\right.
$$
Then $X$ and $Y$ are irreducible real analytic sets such that $m(X,0)=3$ and $m(Y,0)=5$. Moreover, $h$ is a semialgebraic arc-analytic bi-Lipschitz homeomorphism such that $h(X)=Y$.
\end{example}

The next result gives a positive answer to Fukui-Kurdyka-Paunescu's Conjecture.
\begin{theorem}\label{proof_FKP-conjecture}
Let $(X,0)\subset (\R^n,0), (Y,0)\subset (\R^m,0)$ be germs of real analytic sets and let $h\colon (X,0)\to (Y,0)$ be a subanalytic arc-analytic bi-Lipschitz homeomorphism. Then
$m(X,0)\equiv m(Y,0) \,{\rm mod\,} 2$.
\end{theorem}
\begin{proof}

By Proposition \ref{multcone}, for any projection $p\colon \C^{n}\to \C^{d}$ such that $p^{-1}(0)\cap C(X_{\C},0)=\{0\}$, the degree of $C_X'$ with respect to $\pi'$, ${\rm deg}_{\pi'}(C_X')$, is well defined and ${\rm deg}_{\pi'}(C_X')\equiv m(X,0) \,{\rm mod\,} 2$, where $\pi=p|_{\R^n}\colon \R^n\to \R^d$ and $\pi'\colon\mathbb{S}^{n-1}\setminus \pi^{-1}(0)\to \mathbb{S}^{d-1}$ is given by $\pi'(u)=\frac{\pi(u)}{\|\pi(u)\|}$.

Since $C'(X,0)$ is the cone over $C_X'$, the degree of $C'(X,0)$ with respect to $\pi|_{\R^n}\colon \R^n \to \R^d$, ${\rm deg}_{\pi}(C'(X,0))$, is well defined and 
$$
{\rm deg}_{\pi}(C'(X,0)) \equiv {\rm deg}_{\pi'}(C_X')\equiv m(X,0) \,{\rm mod\,} 2.
$$

It follows from Proposition 2.2 in \cite{Valette:2010} that $C_X'$ is an Euler cycle.

\begin{claim}\label{a-invariant}
$C_X'$ is $a$-invariant.
\end{claim}
\begin{proof}[Proof of Claim \ref{a-invariant}]
Let $v\in C_X'\cap {\rm Smp}(\partial X')$. Take an orthogonal projection $\pi\colon \C^n\to \pi(\C^n)\cong\C^d$ such that $\pi(v)=v$ and $\pi^{-1}(0)\cap C(X_{\C},0)=\{0\}$. Thus, $C=\pi^{-1}((-\delta v,\delta v))\cap X$ is an analytic curve. Then, by Lemma 3.3 in \cite{Milnor:1968}, there are an open neighbourhood $U\subset \R^{n}$ of $0$ and $\Gamma_1,\ldots,\Gamma_r\subset\R^{n+1}$ such that $\Gamma_i\cap\Gamma_j=\{0\}$ if $i\not=j$ and  
$$C\cap U=\bigcup\limits _{i=1}^r\Gamma_i.$$ 
Moreover, for each $i\in\{1,\ldots,r\}$, there is an analytic homeomorphism $\gamma_i\colon(-\varepsilon ,\varepsilon )\to \Gamma_i$. 
% For each $i\in\{1,\ldots,r\}$, let $\beta_{2i-1},\beta_{2i}\colon [0,\varepsilon)\to \Gamma_i$ given by $\beta_{2i-1}(t)=\gamma_i(-t)$ and $\beta_{2i}(t)=\gamma_i(t)$.

Since $k_X(v)=\#\{i;\lim\limits_{t\to 0^-}\frac{\gamma_i(t)}{\|\gamma_i(t)\|}=-\lim\limits_{t\to 0^+}\frac{\gamma_i(t)}{\|\gamma_i(t)\|}\}+2\#\{j;\lim\limits_{t\to 0^-}\frac{\gamma_j(t)}{\|\gamma_j(t)\|}=\lim\limits_{t\to 0^+}\frac{\gamma_j(t)}{\|\gamma_j(t)\|}=v\}$ is an odd number, there are an odd number of $i'$s such that $\lim\limits_{t\to 0^-}\frac{\gamma_i(t)}{\|\gamma_i(t)\|}=-\lim\limits_{t\to 0^+}\frac{\gamma_i(t)}{\|\gamma_i(t)\|}$. Then $-v\in C(X,0)$. By the density of ${\rm Smp}(\partial X')$ in $C(X,0)\cap \mathbb{S}^{n-1}$, we can assume that $-v\in {\rm Smp}(\partial X')$. Since $k_X(-v)=\#\{i;\lim\limits_{t\to 0^-}\frac{\gamma_i(t)}{\|\gamma_i(t)\|}=-\lim\limits_{t\to 0^+}\frac{\gamma_i(t)}{\|\gamma_i(t)\|}\}+2\#\{j;\lim\limits_{t\to 0^-}\frac{\gamma_j(t)}{\|\gamma_j(t)\|}=\lim\limits_{t\to 0^+}\frac{\gamma_j(t)}{\|\gamma_j(t)\|}=-v\}$, we find that $-v\in C_X'$. Therefore, $C_X'$ is $a$-invariant.
\end{proof}

% Let $d_0h\colon C(X,0)\to C(Y,0)$ be the mapping given by $d_0h(v)=\lim\limits_{t\to 0^+}\frac{h(\gamma(t)}{t}$, for any subanalytic arc $\gamma\colon[0,\varepsilon)\to X$ such that $\gamma(t)=tv+o(t)$. 
Since $h$ is an arc-analytic bi-Lipschitz homeomorphism, it follows that $d_0h\colon C(X,0)\to C(Y,0)$ is a bi-Lipschitz homeomorphism.

By Theorem 4.2 in \cite{Sampaio:2020c},  $\psi=d_0h|_{C'(X,0)}\colon C'(X,0)\to C'(Y,0)$ is a  bi-Lipschitz homeomorphism and by the proof of Claim \ref{a-invariant}, $\psi$ satisfies $\psi(-v)=-\psi(v)$ whenever $v\in C'(X,0)$.  Thus, $\Gamma={\rm graph}(\psi)$ is an $a$-invariant Euler cycle.

Therefore, the closure $\overline{\Gamma}$ of $\Gamma$ in $\mathbb{P}^{n+m}(\R)$ is a $\mathbb{Z}/(2)$ homological cycle. 
% We finish as in Theorem \ref{inv_degree_homo_cycle}.

We set $G={\rm graph}(\psi^{-1})$. Similarly, the closure $\overline{G}$ of $G$ in $\mathbb{P}^{n+m}(\R)$ is also a $\mathbb{Z}/(2)$ homological cycle.

% \begin{claim}
% For a generic projection $\pi\colon \R^{n+m}\to \R^d$, $\#(\pi^{-1}(x)\cap \Gamma)\,{\rm mod}\, 2$ (resp. $\#(\pi^{-1}(x)\cap G)\,{\rm mod}\, 2$) is constant for a generic $x\in \R^d$. 
% \end{claim}

Since $\Gamma$ is a cone, for any projection $\tilde \pi\colon \R^{n+m}\to \R^d$ such that 
$\tilde \pi^{-1}(0)\cap \Gamma=\{0\}$, ${\rm deg}_{\pi}(\Gamma)$ is defined and equal to the topological degree mod 2 of $[\overline{\Gamma}]$. In fact, since $\tilde \pi^{-1}(0)\cap \Gamma=\{0\}$, for a generic $t\in \R^d$, we have $\tilde \pi^{-1}(t)\cap {\rm Sing}(\Gamma)=\emptyset$ and $\tilde \pi^{-1}(t)$ is transversal to $\Gamma\setminus {\rm Sing}(\Gamma)$, and thus $\#(\tilde \pi^{-1}(t)\cap \Gamma)\,{\rm mod}\, 2$ is equal to the topological degree mod 2 of $[\overline{\Gamma}]$. In particular, ${\rm deg}_{\pi}(\Gamma)$ does not depend on a generic projection $\tilde\pi\colon \R^{n+m}\to \R^d$, and thus ${\rm deg}_2(\Gamma)$ is defined. The same holds true for $[\overline{G}]$, i.e., ${\rm deg}_2(G)$ is defined and is equal to the topological degree mod 2 of $[\overline{G}]$. Since $G= \phi(\Gamma)$, we conclude that ${\rm deg}_2(\Gamma)\equiv {\rm deg}_2(G)\,{\rm mod}\, 2$, where $\phi \colon \R^n\times \R^m\to \R^m\times \R^n$ is given by $\phi(x,y)=(y,x)$.

\begin{claim}
$m(X,0)\equiv {\rm deg}_2(\Gamma)\,{\rm mod}\, 2$.
\end{claim}
\begin{proof}
Let $p\colon \C^{n}\to \C^{d}$ be a linear projection such that $p^{-1}(0)\cap C(X_{\C},0)=\{0\}$. Then, for $\pi=p|_{\R^n}\colon \R^n\to \R^d$, ${\rm deg}_{\pi}(C'(X,0))$ is defined and ${\rm deg}_{\pi}(C'(X,0))\equiv m(X,0) \,{\rm mod\,} 2$. So, it is enough to show that ${\rm deg}_{\pi}(C'(X,0))\equiv {\rm deg}_2(\Gamma) \,{\rm mod\,} 2$.

Let $\tilde \pi \colon \R^n\times \R^m\to \R^d$ be the linear projection given by $\tilde \pi(x,y)=\pi(x)$. Since $\Gamma=C(\Gamma,0)$, we have $\tilde \pi^{-1}(0)\cap \Gamma=\{0\}$. Indeed, if $v=(u,w)\in \tilde \pi^{-1}(0)\cap \Gamma$ then there exists an arc $\gamma\colon[0,\varepsilon)\to \Gamma$ such that $\gamma(t)=tv+o(t)$. Thus, there exists an arc $\beta\colon[0,\varepsilon)\to X$ such that $\gamma(t)=(\beta(t),\psi(\beta(t)))$ for all $t\in [0,\varepsilon)$ and, in particular,  $v=\lim\limits_{t\to 0^+}\frac{\gamma(t)}{t}=(\lim\limits_{t\to 0^+}\frac{\beta(t)}{t},\lim\limits_{t\to 0^+}\frac{\psi(\beta(t))}{t})$. We choose coordinates such that $\pi$ is the projection on the first $d$ coordinates. Thus, for $u=(u_1,u_2)$ and since  $v\in \tilde \pi^{-1}(0)$, we have $u_1=0$. This shows that $u\in \pi^{-1}(0)\cap C(X,0)$, and thus $u=0$. Since $\lim\limits_{t\to 0^+}\frac{\beta(t)}{t}=0$ and there exists a constant $C>0$ such that $\|\frac{\psi(\beta(t))}{t}\|\leq C\|\frac{\beta(t)}{t}\|$ for all $t\in (0,\varepsilon)$, we find that $v=0$.

Thus, ${\rm deg}_{\pi}(\Gamma)\equiv {\rm deg}_2(\Gamma)\,{\rm mod}\, 2$.

Moreover, for a generic $t\in \R^d$, the set $\tilde \pi^{-1}(t)\cap \Gamma=\{(x,y,z)\in \Gamma; x=t $ and $ z=\psi(x,y)\}$ is homeomorphic to $\{(x,y)\in C'(X,0); x=t\}$. Therefore, ${\rm deg}_{\pi}(C'(X,0))\equiv {\rm deg}_2(\Gamma)\,{\rm mod}\, 2$. 
\end{proof}
Similarly, $m(Y,0)\equiv {\rm deg}_2(G)\,{\rm mod}\, 2$ and since ${\rm deg}_2(\Gamma)\equiv {\rm deg}_2(G)\,{\rm mod}\, 2$, we finish the proof.
\end{proof}

\begin{remark}
Theorem \ref{proof_FKP-conjecture} proves even more than stated in Fukui-Kurdyka-Paunescu's Conjecture, since we do not require in Theorem \ref{proof_FKP-conjecture} that the sets $X$ and $Y$ have to be irreducible or that $h$ has to be defined on a neighbourhood of $0\in \R^n$.
\end{remark}

% \begin{proof}
% Let $\tilde X=Graph(h)$ and $\tilde Y=Graph(h^{-1})$. Just as before, $\tilde Y$ is a real analytic set.
% 
% Since $\tilde X$ and $\tilde Y$ real analytic isomorphic, then $m(\tilde X)=m(\tilde Y)$ (see Proposition 4.4 in \cite{Sampaio:2021b}). 
% 
% Let $\pi\colon \C^n\times \C^m\to \C^d$ be a linear projection such that 
% $$
% \pi^{-1}(0)\cap (C(\tilde X_{\C},0)\cup C((X\times \{0\})_{\C},0))=\{0\}.
% $$
% \end{proof}

\subsection{On the global version of Fukui-Kurdyka-Paunescu's Conjecture}

\begin{theorem}\label{thm:arc-symmetric}
Let $A\subset \R^n, B\subset \R^m$ be  real algebraic $d$-dimensional sets and let $h\colon A\to B$ be a semialgebraic and semi-bi-Lipschitz homeomorphism. If the graph of $h$ is arc-symmetric, then
${\rm deg}(A)\equiv {\rm deg}(B)\,{\rm mod\,} 2.$
\end{theorem}

\begin{proof}

Let $\Gamma ={\rm graph}(h)$. Let $\overline{\Gamma}$ be the closure  of $\Gamma$ in $\mathbb{P}^{n+m}(\R)$. Let $Z$ be the Zariski closure of $\overline{\Gamma}\setminus \Gamma$ in $\mathbb{P}^{n+m}(\R).$ Since the arc-symmetric sets form a constructible category (\cite{PA}), we find that $\Gamma':=\Gamma\cup Z$ is arc-symmetric. 

 Take a semi-algebraic triangulation $\mathcal S$ of $\Gamma'$ such that the set $Z$ is a union of strata. Hence all $d$-dimensional cells of this stratification are contained in $\R^{n+m}.$ On $\overline{\Gamma}$ we have the induced stratification $\mathcal S'.$ Now every $(d-1)$-dimensional cell $C$ in $\mathcal S'$ comes from $\mathcal S.$ Since the set $\Gamma'$ is arc-symmetric it is an Euler cycle, see \cite{PA}. This means that $C$ meets an even number of $d-$dimensional cells. But every $d-$dimensional cell is contained in $\Gamma.$ Finally $C$ meets an even number of $d-$dimensional cells in $\mathcal S'.$ This means that $\overline{\Gamma}$ is a ${\mathbb{Z}}/{(2)}$ homological cycle. By Theorem \ref{Ztwo_cycle}, ${\rm deg}(A)\equiv {\rm deg}(B)\,{\rm mod\,} 2.$
\end{proof}

\begin{definition}
Let $X\subset \R^n$ and $Y\subset \R^m$ be analytic subsets and let $\overline{X}$ be the closure of $X$ in $\mathbb{P}^n(\R)$. We say that a mapping $f\colon X\to Y$ is {\bf arc-analytic at $\overline{X}$} if for any analytic arc $\gamma\colon (-1,1)\to \mathbb{P}^n(\R)$ such that $\gamma((-1,0)\cup (0,1))\subset X$, we have $f\circ \gamma|_{(-1,0)\cup (0,1)}$ extends to an
analytic arc $\tilde\gamma\colon (-1,1)\to \mathbb{P}^m(\R)$.
\end{definition}

\begin{corollary}\label{proof_global_FKP}
Let $A\subset \R^n, B\subset \R^m$ be real algebraic $d-$dimensional sets, let $\overline{A}$ be the closure of $A$ in $\mathbb{P}^n(\R)$ and let $h\colon A\to B$ be a semialgebraic and semi-bi-Lipschitz homeomorphism. Assume that $h$ is arc-analytic at $\overline{A}$. Then
${\rm deg}(A)\equiv {\rm deg}(B)\,{\rm mod\,} 2.$
\end{corollary}

\section{Invariance of the degree and multiplicity of complex analytic sets}\label{section:mainresults}
% \subsection{Invariance of the multiplicity}

\subsection{Invariance of the multiplicity}

The next result generalizes Theorem 4.1 in \cite{Jelonek:2021}.
\begin{theorem}\label{inv_mult}
Let $(X,0)\subset (\C^n,0), (Y,0)\subset (\C^m,0)$ be germs of complex analytic sets and let $h\colon (X,0)\to (Y,0)$ be a germ of homeomorphism which is also semi-bi-Lipschitz at $0$. Assume that the graph of $h$ is a complex analytic set. Then 
$m(X,0)= m(Y,0).$
\end{theorem}
\begin{proof}
Let $U,V$ be small neighborhoods of $0$ in $\C^n$ and $\C^m$ such that the mapping $h:U\cap X=X'\to V\cap Y=Y'$ is defined and it is semi-bi-Lipschitz.
Denote by $\Gamma\subset U\times V$ the graph of $h.$ By Lemma \ref{wykres} the projections $\pi_{X'}:\Gamma\to X'$ and $\pi_{Y'}:\Gamma\to Y'$ are   semi-bi-Lipschitz homeomorphism. Let
$\pi_1\colon \C^n\times \C^m\to \C^n$ and $\pi_2\colon \C^n\times \C^m\to \C^m$ be projections and denote by $S_1,S_2\subset \pi_\infty=\mathbb{P}^{n+m-1}(\C)$ the centers of these projections. Denote by $\Lambda_0\subset \pi_\infty$ the set of directions of all secants of $\Gamma$ which contain $0$ and let $\Sigma_0= cl(\Lambda_0)$.
Since $\pi_1|_\Gamma=\pi_X$ and $\pi_2|_\Gamma=\pi_Y$ we see by Lemma \ref{lemat} that $\Sigma_0 \cap S_1=\Sigma\cap S_2=\emptyset.$ 
Let $C(\Gamma,0)$ denote the tangent cone of $\Gamma$ at $0$. We have  $\overline{C(\Gamma,0)}\setminus C(\Gamma,0)\subset \Sigma_0$ and consequently $\overline{C(\Gamma,0)}\cap S_i=\emptyset$ for $i=1,2.$ Now let $L\subset \C^n$ be a generic linear subspace of dimension $k={\rm codim } \ X.$ Then $\#(L\cap X')={\rm mult}_0 X$ and  $L$ has no common points with $C(X,0)$ at infinity (we can shrink $U,V$ if necessary!).
This implies that also $(\overline{C(\Gamma,0)}\setminus C(\Gamma,0))\cap \langle S_1,L\rangle =\emptyset$ and consequently $\#(\langle S_1,L\rangle \cap\ \Gamma)={\rm mult}_0\ \Gamma$, where $\langle S_1,L\rangle $ is the linear (projective) subspace spanned by $L$ and $S_1.$ However the mapping $\pi_{X'}$ is a bijection, hence $\#(\langle S_1,L\rangle \cap\ \Gamma)={\rm mult}_0\ X.$ In particular ${\rm mult}_0\ \Gamma={\rm mult}_0\ X.$ In the same way ${\rm mult}_0\ \Gamma={\rm mult}_0\ Y.$ Hence  ${\rm mult}_0\ X={\rm mult}_0 \ Y.$

\end{proof}

\subsection{Invariance of the degree}

\begin{theorem}\label{inv_degree_semi}
Let $X\subset \C^n,Y\subset \C^m$ be complex algebraic sets and let $h\colon X\to Y$ be a mapping. Assume that $h$ is semi-bi-Lipschitz at infinity and its graph is a complex algebraic set. Then
${\rm deg} (X)= {\rm deg }(Y).$
\end{theorem}

\begin{proof}
This follows directly from Theorem \ref{inv_degree_homo_cycle}.
\end{proof}

% \begin{remark}
% In Theorem \ref{inv_degree_semi}, if $h$ is bi-Lipschitz at infinity, we only have to ask that its graph is a complex analytic set (see \cite[Theorem 3.1]{Sampaio:2021}).
% \end{remark}

\begin{theorem}\label{inv_degree_derivative}
Let $X\subset \C^n,Y\subset \C^m$ be complex algebraic sets and let $h\colon X\to Y$ be a semialgebraic and semi-bi-Lipschitz homeomorphism. Assume that $d_{\infty}h$ is $\C$-homogeneous. Then
${\rm deg} (X)= {\rm deg } (Y).$
\end{theorem}

\begin{proof}
We can extend the mapping $h$ to the infinity -- we simply take a path $a(t)=wt+o_{\infty}(t)$ which tends to the point $[w]\in \mathbb{P}^{n}(\C)$ and by semilinearity the limit $\lim\limits_{t\to \infty} f(a(t))$ does not depend on $w$ but only on $[w]$. Hence we have an induced
homeomorphism $\bar h\colon \overline X\to G$, where $\overline X$ is the closure of $X$ in $\mathbb{P}^{n}(\C)$ and $G$ is the closure of ${\rm graph}(h)$ in $\mathbb{P}^{n+m}(\C)$. Hence $G$ is an orientable homological cycle. Then the conclusion follows from Theorem \ref{inv_degree_homo_cycle}.
\end{proof}

\subsection{Fukui-Kurdyka-Paunescu's Conjecture in the complex case}
The next result is a direct consequence of Theorem \ref{proof_FKP-conjecture}, but we present a proof  which is a little easier.
\begin{proposition}\label{inv_mult_homo_cycle_mod_two}
Let $X\subset \C^n,Y\subset \C^m$ be complex analytic sets and let $h\colon (X,0)\to (Y,0)$ be a subanalytic arc-analytic bi-Lipschitz homeomorphism. Then
$m(X,0)\equiv m(Y,0)\, {\rm mod}\, 2$.
\end{proposition}
\begin{proof}
Let $\psi=d_0h\colon C(X,0)\to C(Y,0)$ be the pseudo-derivative of $h$ at $0$. Let $A_1,\ldots,A_r$ be the irreducible components of $C(X,0)$. Thus, $B_1=\psi(A_1),\ldots,B_r=\psi(A_r)$ are the irreducible components of $C(Y,0)$.

% \begin{claim}
Since $h$ is arc-analytic, $\Gamma_i={\rm graph}(\psi_i)$ is an $a$-invariant Euler cycle, where $\psi_i=\psi|_{A_i}$.
% \end{claim}
Therefore, the closure $\overline{\Gamma}_i$ of $\Gamma_i$ in $\mathbb{P}^{2(n+m)}(\R)$ is a $\mathbb{Z}/(2)$ homological cycle. By Theorem \ref{Ztwo_cycle}, ${\rm deg} (A_i)= {\rm deg } (B_i) \ {\rm mod} \ 2$. Since ${\rm deg} (A_i)=m(A_i,0)$ and ${\rm deg} (B_i)=m(B_i,0)$, we obtain
$m(A_i,0)\equiv m(B_i,0)\, {\rm mod}\, 2$, for each $i\in \{1,\ldots,r\}$. Moreover, $m(X,0)=\sum\limits_{i=1}^rk_X(A_i)m(A_i,0)$ and $m(Y,0)=\sum\limits_{i=1}^rk_Y(B_i)m(B_i,0)$. By Proposition 1.6 in \cite{FernandesS:2016}, $k_X(A_i)=k_Y(B_i)$ for each $i\in \{1,\ldots,r\}$. Therefore, 
$m(X,0)\equiv m(Y,0)\, {\rm mod}\, 2$.
\end{proof}

\bigskip

\noindent{\bf Acknowledgements}. The authors are grateful to Professor Wojciech Kucharz from Jagiellonian University at Krak\'ow for helpful conversations.

\end{document}